\documentclass[11pt,a4paper]{amsart}
\usepackage[T1]{fontenc}
\usepackage{lmodern}
\usepackage[british]{babel}
\usepackage{a4wide}
\usepackage{amsmath,amssymb}
\usepackage{dsfont}
\usepackage{url}
\usepackage[colorlinks=true,linkcolor=blue,citecolor=blue,urlcolor=blue]{hyperref}

\newtheorem{maintheorem}{Theorem}

\newtheorem{theorem}{Theorem}[section]
\newtheorem{lemma}[theorem]{Lemma}

\theoremstyle{definition}
\newtheorem{definition}[theorem]{Definition}
\theoremstyle{remark}
\newtheorem{question}[theorem]{Question}
\newcommand{\F}{\mathcal F}
\newcommand{\Lipzero}{\operatorname{Lip}_0}
\newcommand{\N}{\mathbb N}
\newcommand{\R}{\mathbb R}
\newcommand{\Z}{\mathbb Z}
\newcommand{\norm}[1]{\left\lVert #1\right\rVert}
\newcommand{\diam}{\operatorname{diam}}
\newcommand{\dist}{\operatorname{dist}}
\newcommand{\dimA}{\dim_{\mathrm A}}

\title{Kwapie\'n's property (K) in Lipschitz-free spaces}
\author[T.~Kania]{Tomasz Kania}
\address[T.~Kania]{Mathematical Institute\\Czech Academy of Sciences\\
\v Zitn\'a 25\\115 67 Praha 1\\Czech Republic and
Institute of Mathematics and Computer Science\\Jagiellonian University\\
{\L}ojasiewicza 6, 30-348 Krak\'{o}w, Poland}
\email{kania@math.cas.cz, tomasz.marcin.kania@gmail.com}
\thanks{RVO: 67985840.}
\author[J.~K{\k a}kol]{Jerzy K{\k a}kol}
\address[J.~K{\k a}kol]{Department of Functional Analysis\\
Faculty of Mathematics and Computer Science\\
Adam Mickiewicz University\\Uniwersytetu Pozna\'nskiego 4\\
61-614 Pozna\'n, Poland}
\email{kakol@amu.edu.pl}
\date{}
\hypersetup{pdftitle={Kwapie\'n's property (K) in Lipschitz-free spaces},
  pdfauthor={Tomasz Kania and Jerzy K{\k a}kol}}

\begin{document}
\begin{abstract}
We prove that a Lipschitz-free space $\F(M)$ fails Kwapie\'n's property (K)
whenever its underlying metric space $M$ contains a bi-Lipschitz copy of
$C\times K$, where $C\subset\R$ is closed with positive measure and
$K$ contains a Cantor metric space. Consequently, this holds for every subset
of $\R^n$, $n\geqslant2$, whose closure has positive measure, and for a
compact planar set of Assouad dimension one. These Lipschitz-free spaces
embed into no strongly weakly compactly generated space and no $L_1$-space;
in the separable case their duals are not isomorphic to $\ell_\infty$.
Combining our result with a decomposition of Aliaga and Medina, we show
that $\F(\R^n\times\Z^m)$ is strongly weakly compactly generated if and
only if it has (K), and this occurs precisely when $n\leqslant1$.
In particular, $\F(\R^2)$ does not embed isomorphically into
$\F(\R\times\Z)$, answering a question of Aliaga. The planar examples
nevertheless have the Dunford--Pettis property and Pe{\l}czy\'nski's
property $(V^*)$ and are weakly sequentially complete.
\end{abstract}
\maketitle

\section{Introduction}

Kwapie\'n's property (K), in the formulation of Kalton and
Pe{\l}czy\'nski \cite[Section~2]{KP}, concerns convex blocks of weak*-null
sequences in a dual Banach space. A Banach space $X$ has property (K)
if every such sequence admits a convex block subsequence converging to
zero uniformly on each weakly compact subset of $X$. Equivalently, the
blocks converge to zero in the Mackey topology $\mu(X^*,X)$.
Avil\'es and Rodr\'iguez \cite[Corollary~2.3]{AR} show that every subspace
of a strongly weakly compactly generated (SWCG) space has (K). Finite-measure
$L_1$-spaces are strongly weakly compactly generated \cite{SW}, so their
subspaces have (K). Kwapie\'n's argument in Kalton and
Pe{\l}czy\'nski \cite[Section~2]{KP} already shows that finite-measure
$L_1$-spaces have (K).

The Lipschitz-free space of $I=[0,1]$ is isometric to $L_1[0,1]$
\cite[Corollary~3.3]{Godard} and
therefore has (K). We show that already a Cantor family of suitable
subsets of an interval forces failure of (K). The proof constructs a
uniformly null, uniformly Lipschitz sequence whose every convex block
subsequence is detected on one fixed fibre. The fibre depends on the
whole block sequence. No linear extension operator on the ambient
metric space is required.

All Banach spaces are real, $\N=\{1,2,\ldots\}$, and $\cong$ denotes a
Banach space isomorphism. A \emph{Cantor metric space} is a metric space homeomorphic
to $\{0,1\}^{\N}$. On $C\times K$, where $C\subset\R$ and
$(K,d_K)$ is a metric space, we use the metric
\[
 d_2((x,t),(y,u))=\bigl(|x-y|^2+d_K(t,u)^2\bigr)^{1/2}.
\]
Equivalent standard product metrics give the same conclusions.
We equip $\R^n\times\Z^m$ with the Euclidean metric inherited from
$\R^{n+m}$. Any other norm on $\R^{n+m}$ gives an isomorphic
Lipschitz-free space, since the identity map is bi-Lipschitz.

\begin{maintheorem}\label{thm:product}
Suppose that a metric space $M$ contains a bi-Lipschitz copy of
$C\times K$, where $C\subset\R$ is closed and has positive Lebesgue
measure, and $K$ contains a Cantor metric space. Then $\F(M)$ fails
property (K) and is isomorphic to no subspace of a strongly weakly
compactly generated space or of any $L_1$-space. If $M$ is separable,
then $\Lipzero(M)\not\cong\ell_\infty$.
\end{maintheorem}

In particular, the theorem applies to $I\times K$ for every uncountable
compact metric space $K$, to every cube $I^n$ with $n\geqslant2$, and to
every normed space of dimension at least two. Indeed,
every uncountable compact metric space contains a Cantor subspace.

An application to the isomorphic classification of Lipschitz-free
spaces follows from Aliaga and Medina's decomposition
\cite[Proposition~3.14]{AM}: for $n,m\geqslant1$, the space
$\F(\R^n\times\Z^m)$ is isomorphic to the direct sum of
$\F(\R^n)$ and $\F(\Z^{n+m})$. Their proof uses the maximum norm;
the preceding norm equivalence gives the Euclidean formulation used
here. Throughout, $\R^0$ and $\Z^0$ denote singletons.

\begin{maintheorem}\label{thm:bundles}
Let $n,m\geqslant0$ be integers. The following conditions are equivalent:
\begin{enumerate}
\item $\F(\R^n\times\Z^m)$ is strongly weakly compactly generated;
\item $\F(\R^n\times\Z^m)$ has property (K);
\item $n\leqslant1$.
\end{enumerate}
Moreover, if $d\geqslant2$ and $m\geqslant0$, then $\F(\R^d)$ does not
embed isomorphically into $\F(\R\times\Z^m)$.
\end{maintheorem}

In particular, $\F(\R\times\Z)\not\cong\F(\R^2)$. This answers
\cite[Question~5]{Aliaga}, also recorded in \cite[Section~3.5]{AM}
and \cite[Problem~2]{Mason}. Mason's recent Dunford--Pettis theorem
\cite[Theorem~1 and Corollary~2]{Mason} shows that both spaces have
the Dunford--Pettis property, so that property cannot distinguish them.
Property (K) does, and the SWCG conclusion yields the stronger
nonembeddability result.
In the original topological formulation of SWCG, Theorem~\ref{thm:bundles}
also says that the closed unit ball of $\Lipzero(\R^n\times\Z^m)$
is metrisable in its Mackey topology precisely when $n\leqslant1$.

An Eggleston-type product lemma gives a further geometric consequence.

\begin{maintheorem}\label{thm:measure}
Every subset $E\subset\R^n$ whose closure has positive Lebesgue measure,
where $n\geqslant2$, satisfies all the conclusions of
Theorem~\ref{thm:product}.
\end{maintheorem}

Naor and Schechtman \cite{NS} proved that the $L_1$-distortion of planar
earthmover metrics on square grids tends to infinity, with a lower bound
of order $\sqrt{\log n}$ for grids of side length $n$. Their theorem
implies that $\F(I^2)$ and $\F(\R^2)$ do not embed into $L_1$.
This also implies
$\Lipzero(I^2)\not\cong\ell_\infty$: otherwise $\F(I^2)$ would embed
into $\ell_\infty^*$ and hence, by the measure-domination argument in
Lemma~\ref{lem:Lone}, into a finite-measure $L_1$-space.
Theorem~\ref{thm:product} gives a direct qualitative proof of these
consequences. Its additional scope includes arbitrary positive-measure
sets, products $I\times K$ with null Cantor sets $K$, and the example
below, where square grids of uniformly bounded distortion are unavailable.

\begin{maintheorem}\label{thm:assouad}
There is a compact planar set $M$ with $\dimA M=1$ satisfying all the
conclusions of Theorem~\ref{thm:product}. For each fixed distortion
bound, $M$ contains bi-Lipschitz copies of square grids of only bounded
size, even when arbitrary rescaling is allowed.
The space $\F(M)$ has the Dunford--Pettis property and property $(V^*)$
and is weakly sequentially complete.
\end{maintheorem}

The Euclidean examples above combine failure of (K) with the Dunford--Pettis
property, property $(V^*)$ and weak sequential completeness
\cite{KoP,APQ,Mason}. By contrast, complete purely $1$-unrectifiable
metric spaces have Schur Lipschitz-free spaces \cite[Theorem~C]{AGPP}
and hence have (K); neither these Schur spaces nor our weakly
sequentially complete examples contain $c_0$.

\section{Preliminaries}

For a pointed metric space $(M,o)$, let $\Lipzero(M)$ be the Banach
space of real Lipschitz functions vanishing at $o$, with norm
$\norm{f}_{\mathrm{Lip}}$. Its canonical predual is the
\emph{Lipschitz-free space}
\[
 \F(M)=\overline{\operatorname{span}}\{\delta_p:p\in M\}
 \subset\Lipzero(M)^*,\qquad \langle f,\delta_p\rangle=f(p).
\]
Changing the base point gives canonically isometric spaces. McShane's
extension theorem \cite[Theorem~1.33]{Weaver} identifies $\F(N)$ isometrically with a closed
subspace of $\F(M)$ whenever $N\subset M$ contains the base point.
Moreover, $\F(M)$ is canonically isometric to $\F(\widehat M)$, where
$\widehat M$ is the metric completion. A uniformly bounded sequence in
$\Lipzero(M)$ is weak*-null if it converges pointwise to zero, by density
of the evaluation functionals.

For a Banach space $X$, write $B_X$ for its closed unit ball. The
\emph{Mackey topology} $\mu(X^*,X)$ on $X^*$ is the locally convex
topology of uniform convergence on weakly compact subsets of $X$.
It is generated by the seminorms
\[
 p_W(f)=\sup_{x\in W}|\langle f,x\rangle|,
 \qquad W\subset X\text{ weakly compact}.
\]
One may equivalently restrict to absolutely convex weakly compact
sets, since the closed absolutely convex hull of a weakly compact
set is weakly compact. Thus $g_k\to0$ in $\mu(X^*,X)$ means that
$p_W(g_k)\to0$ for every such $W$. The weak* topology
$\sigma(X^*,X)$ requires only pointwise convergence on $X$.

A convex block subsequence of $(f_m)$ has the form
\begin{equation}\label{eq:blocks}
 g_k=\sum_{m\in E_k}a_m f_m,\qquad
 a_m\geqslant0,\quad\sum_{m\in E_k}a_m=1,
 \quad\max E_k<\min E_{k+1},
\end{equation}
where the $E_k\subset\N$ are finite and non-empty. The following
standard consequence of Eberlein--\v Smulian connects the Mackey
formulation of (K) with the pairing formulation in \cite{KP}.
We recall the proof for completeness; see also the quantitative
version in \cite[Lemma~3.1]{CKR}.

\begin{lemma}\label{lem:diagonal}
A weak*-null sequence $(g_k)$ in $X^*$ converges to zero in
$\mu(X^*,X)$ if and only if
$\langle g_k,x_k\rangle\to0$ for every weakly null sequence $(x_k)$ in $X$.
\end{lemma}

\begin{proof}
One implication follows by applying uniform convergence to the weakly
compact set $\{0,x_1,x_2,\ldots\}$. Conversely, failure of uniform
convergence on a weakly compact set $W$ gives $\varepsilon>0$, indices
$k_j$ and points $x_j\in W$ with
$|\langle g_{k_j},x_j\rangle|\geqslant\varepsilon$. By
Eberlein--\v Smulian, pass to a subsequence such that $x_j\to x$ weakly.
Set $z_{k_j}=x_j-x$ and $z_k=0$ at all other indices. Then $(z_k)$ is
weakly null, whereas weak* convergence gives
$|\langle g_{k_j},z_{k_j}\rangle|\geqslant\varepsilon/2$ eventually.
\end{proof}

The proof of Theorem~\ref{thm:product} uses only the forward implication.
The converse reconciles the two formulations of property (K).

\begin{definition}\label{def:SWCG}
A Banach space $Y$ is \emph{strongly weakly compactly generated}
(SWCG) if the relative Mackey topology $\mu(Y^*,Y)$ on $B_{Y^*}$
is metrisable.
\end{definition}

This is the original definition of Schl\"uchtermann and Wheeler.
By \cite[Theorem~2.1]{SW}, it is equivalent to the existence of a
weakly compact absolutely convex set $H\subset Y$ such that
\begin{equation}\label{eq:strong-generation}
 \forall W\subset Y\text{ weakly compact}\quad
 \forall\varepsilon>0\quad\exists n\in\N:
 \quad W\subset nH+\varepsilon B_Y.
\end{equation}
Here $nH=\{nh:h\in H\}$, and $H$ is said to \emph{strongly generate}
$Y$. We use this characterisation below. Equivalently, there is a
sequence $(A_j)$ of weakly compact absolutely convex subsets of $Y$
such that, for every weakly compact $W\subset Y$ and every
$\varepsilon>0$, some $A_j$ satisfies \mbox{$W\subset A_j+\varepsilon B_Y$}.
This expresses the weakly compact density condition: a countable
family approximates every weakly compact set to arbitrary norm
accuracy. The three equivalent formulations are also recorded in
\cite[Theorem~10]{KLP}; for the density-condition viewpoint, see
\cite[Section~6.7, pp.~177--183]{KKLS}.

The result below for separable Schur spaces is due to
Schl\"uchtermann and Wheeler \cite{SW}; see also
Mercourakis and Stamati \cite[Example~1.3(ii)]{MS}.
We include its short proof and the elementary permanence properties
needed for Theorem~\ref{thm:bundles}.

\begin{lemma}\label{lem:SWCG}
Every separable Schur space is SWCG. Finite direct sums of SWCG spaces
are SWCG, and the property is invariant under isomorphisms.
\end{lemma}

\begin{proof}
Let $Y$ be a non-zero separable Schur space and let $(y_j)$ be dense
in its unit sphere. The closed absolutely convex hull $H$ of
$\{2^{-j}y_j:j\in\N\}$ is norm compact and has dense linear span.
Since $H$ is absolutely convex, $\bigcup_{k\geqslant1}kH$ is norm
dense in $Y$. Every weakly compact set $W\subset Y$ is norm compact,
by the Schur property and Eberlein--\v Smulian. A finite
$\varepsilon/2$-net in $W$, followed by approximation of its points
in $\bigcup_k kH$, gives $W\subset kH+\varepsilon B_Y$ for some $k$.
The zero space is immediate. For a finite direct sum with the maximum
norm, take the product
of weakly compact absolutely convex generating sets and apply the
inclusions in \eqref{eq:strong-generation} to the coordinate projections of $W$.
Finally, an isomorphism transports a generating set and
changes only the required approximation tolerance.
\end{proof}

SWCG itself does not pass to closed subspaces: Mercourakis and
Stamati \cite[Corollary~3.10]{MS} construct a non-SWCG closed subspace
of $L_1[0,1]$. Thus the nonembedding conclusion in
Theorem~\ref{thm:product} uses the stronger fact that every subspace
of an SWCG space has (K), due to Avil\'es and Rodr\'iguez
\cite[Corollary~2.3]{AR}. Together with the finite-measure $L_1$
case of \cite{SW}, this also gives the following observation for
arbitrary measures and dual isomorphisms.

\begin{lemma}\label{lem:Lone}
Every separable closed subspace of any $L_1$-space is isomorphic to a subspace
of a finite-measure $L_1$-space, and hence has (K). If a separable Banach
space $X$ satisfies $X^*\cong\ell_\infty$, then $X$ has (K).
\end{lemma}

\begin{proof}
For a dense sequence $(u_j)$ in a separable subspace of $L_1(\mu)$, put
$w=\sum_{j\geqslant1}2^{-j}|u_j|/(1+\norm{u_j}_1)$ and
$d\nu=w\,d\mu$. Every element of that subspace vanishes off $\{w>0\}$,
and $u\mapsto u/w$, defined as zero when $w=0$, is an isometry into
$L_1(\nu)$, where $\nu$ is finite.

For the second part, identify $X$, up to isomorphism, with a
separable subspace $Y$ of $X^{**}\cong\ell_\infty^*=M(\beta\N)$.
For a dense sequence $(\nu_j)$ in $Y$, the finite positive measure
$\nu=\sum_{j\geqslant1}2^{-j}|\nu_j|/(1+\norm{\nu_j})$ dominates every
$\nu_j$. Absolute continuity is preserved under variation-norm limits,
so Radon--Nikod\'ym derivatives embed $Y$ isometrically into $L_1(\nu)$.
The subspace theorem for (K) completes the proof.
\end{proof}

The derivative map identifies $\Lipzero(I)$ isometrically with
$L_\infty(I)$; Pe{\l}czy\'nski's theorem \cite{Pelczynski} gives
$L_\infty(I)\cong\ell_\infty$. We shall also use that Schur spaces have
(K): their weakly compact sets are norm compact, and bounded weak*-null
sequences converge uniformly on norm-compact sets.
Recall also that the Dunford--Pettis property is equivalent to
$\langle x_j^*,x_j\rangle\to0$ whenever both $(x_j)\subset X$ and
$(x_j^*)\subset X^*$ are weakly null. In contrast, the dual sequences
in the definition of (K) are only assumed to be weak*-null.

\section{The product construction}

Passing to a compact subset and rescaling, it suffices to consider
compact $C\subset I$ with $\min C=0$, $\max C=1$ and $c=|C|>0$.
When normalising $C$, rescale the metric on $K$ by the same factor;
this rescaling leaves the distortion of the embedding unchanged.
Let
\[
 r_j(x)=(-1)^{\lfloor2^jx\rfloor}\quad(0\leqslant x<1),
 \qquad H_j(x)=\int_0^x\mathds{1}_C(s)r_j(s)\,ds.
\]
The Rademacher functions are orthonormal in $L_2(I)$. By Bessel's
inequality, applied to $\mathds{1}_{C\cap[0,x]}$, we have $H_j(x)\to0$
for every $x$, and their common Lipschitz bound $1$ implies
\begin{equation}\label{eq:primitive}
 \norm{H_j}_\infty\longrightarrow0.
\end{equation}

\begin{lemma}\label{lem:fibre}
There is a weakly null sequence $(\nu_j)$ in the unit sphere of
$\F(C)$ such that $\langle H_j|_C,\nu_j\rangle=1$ for all $j$.
\end{lemma}

\begin{proof}
We first describe the affine extension used in
\cite[Theorem~3.2]{Godard}. Since $C$ is compact and contains $0$ and $1$,
each component of $I\setminus C$ is an open interval $(a,b)$ with
$a,b\in C$. For $h\in\Lipzero(C)$, define $\Lambda h=h$ on $C$ and
\[
 (\Lambda h)(x)
 =\frac{b-x}{b-a}h(a)+\frac{x-a}{b-a}h(b)
 \qquad(a<x<b).
\]
This defines a linear operator $\Lambda:\Lipzero(C)\to\Lipzero(I)$.
On $[a,b]$, the slope of $\Lambda h$ is
$(h(b)-h(a))/(b-a)$, whose modulus is at most
$\norm{h}_{\mathrm{Lip}}$. If $x<y$ do not lie in the closure of one gap,
let $a$ and $b$ be the first and last points of $C\cap[x,y]$.
Applying the preceding estimate on $[x,a]$ and $[b,y]$, and the Lipschitz
bound for $h$ on $[a,b]$, gives
\[
 |(\Lambda h)(y)-(\Lambda h)(x)|
 \leqslant\norm{h}_{\mathrm{Lip}}
 \bigl((a-x)+(b-a)+(y-b)\bigr)
 =\norm{h}_{\mathrm{Lip}}(y-x).
\]
Since restriction to $C$ recovers $h$, $\Lambda$ is an isometry.
Its adjoint sends an evaluation at $x\in C$ to the same evaluation on
$C$, and for $x\in(a,b)\subset I\setminus C$ it satisfies
\[
 \Lambda^*\delta_x
 =\frac{b-x}{b-a}\delta_a+\frac{x-a}{b-a}\delta_b,
\]
where the evaluations on the right belong to $\F(C)$.
Thus $\Lambda^*$ maps the linear span of the evaluations on $I$ into
$\F(C)$ and, by continuity, restricts to a contraction
$P:\F(I)\to\F(C)$. The identity
$\langle h,P\mu\rangle=\langle\Lambda h,\mu\rangle$
for $h\in\Lipzero(C)$ and $\mu\in\F(I)$ gives $P^*=\Lambda$.

For $H\in\Lipzero(I)$, write $H'$ for its almost-everywhere derivative.
Such a function is absolutely continuous,
$H(x)=\int_0^x H'(s)\,ds$, and
$\norm{H'}_\infty=\norm{H}_{\mathrm{Lip}}$.
The interval representation in \cite[Corollary~3.3]{Godard} can therefore
be described directly by
\[
 \langle H,Ju\rangle=\int_0^1H'(x)u(x)\,dx
 \qquad(H\in\Lipzero(I),\ u\in L_1(I)).
\]
Indeed, $J\mathds{1}_{(a,b]}=\delta_b-\delta_a$ whenever
$0\leqslant a<b\leqslant1$. Interval step functions are dense in $L_1(I)$,
so the estimate $\norm{Ju}\leqslant\norm{u}_1$ shows that
$Ju\in\F(I)$ for every $u\in L_1(I)$.
Taking $H(x)=\int_0^x\operatorname{sgn}u(s)\,ds$ gives the reverse
inequality, hence $J:L_1(I)\to\F(I)$ is an isometry.
Its range is closed and contains every evaluation
$\delta_x=J\mathds{1}_{(0,x]}$, so it is onto.

Put $\nu_j=c^{-1}PJ(\mathds{1}_Cr_j)$.
Since $\norm{\mathds{1}_Cr_j}_1=c$, we have $\norm{\nu_j}\leqslant1$.
For every $h\in\Lipzero(C)$,
\[
 \langle h,\nu_j\rangle
 =c^{-1}\int_C r_j(x)(\Lambda h)'(x)\,dx\longrightarrow0
\]
by Bessel's inequality, since $\mathds{1}_C(\Lambda h)'\in L_2(I)$.
Thus $(\nu_j)$ is weakly null.
The function $H_j$ is constant on every gap of $C$, so
$\Lambda(H_j|_C)=H_j$. Consequently
\[
 \langle H_j|_C,\nu_j\rangle
 =c^{-1}\int_C r_j(x)^2\,dx=1.
\]
As $\norm{H_j|_C}_{\mathrm{Lip}}\leqslant1$, this also gives
$\norm{\nu_j}=1$.
\end{proof}

Let $K$ be a Cantor metric space. Fix a homeomorphism
$S\mapsto t_S$ from the subsets of $\N$, identified with
$\{0,1\}^{\N}$, onto $K$. For a finite binary word $\sigma$, define
the compact cylinder
\[
 K_\sigma=\{t_S:\mathds{1}_S|_{\{1,\ldots,|\sigma|\}}=\sigma\}.
\]
For $b\geqslant1$, put
\[
 d_b=\min_{\substack{|\sigma|=|\tau|=b\\\sigma\ne\tau}}
       \dist(K_\sigma,K_\tau)>0,
 \qquad \varepsilon_b=\min(2^{-b},d_1/2,\ldots,d_b/2).
\]
The sequence $(\varepsilon_b)$ decreases to zero. By
\eqref{eq:primitive}, choose strictly increasing indices $j_b$ with
$\norm{H_{j_b}}_\infty\leqslant\varepsilon_b$.
For a fixed $m$ and a subset $S\subset\N$, call the elements of
$S\cap\{m,m+1,\ldots\}$ the selected indices. Define on $C\times K$
\begin{equation}\label{eq:construction}
 f_m(x,t_S)=
 \begin{cases}
 H_{j_b}(x),& b=\min(S\cap\{m,m+1,\ldots\})\text{ exists},\\
 0,& S\cap\{m,m+1,\ldots\}=\varnothing.
 \end{cases}
\end{equation}
These functions vanish on $\{0\}\times K$, satisfy
$\norm{f_m}_\infty\leqslant\varepsilon_m$, and are $1$-Lipschitz in
the first coordinate. To check the second coordinate, suppose that
$S,U$ first differ at $b$. If $b\geqslant m$ and a selected index
lies in $[m,b)$, the first selected indices agree and the values in
\eqref{eq:construction} coincide. Otherwise all selected indices are
at least $b$; this also holds when $b<m$, since they are then at least
$m>b$. The difference is therefore at most
$2\varepsilon_b\leqslant d_b\leqslant d_K(t_S,t_U)$.
Going through $(x',t)$ now gives the explicit estimate
\begin{equation}\label{eq:productbound}
 |f_m(x,t)-f_m(x',u)|
 \leqslant |x-x'|+d_K(t,u)
 \leqslant\sqrt2\,d_2((x,t),(x',u)).
\end{equation}

\begin{proof}[Proof of Theorem~\ref{thm:product}]
Reduce $C$ to the compact normalised case above and $K$ to a Cantor
subspace. Let $\Phi:C\times K\to M$ be a bi-Lipschitz embedding, and
choose the base point $o=\Phi(0,t_\varnothing)$. Write $d_M$ for the
metric on $M$, put $Q=\Phi(C\times K)$, and define
$g_m=f_m\circ\Phi^{-1}$ on $Q$. By \eqref{eq:productbound}, these
functions have a common Lipschitz bound
$L=\sqrt2\,\norm{\Phi^{-1}}_{\mathrm{Lip}}$.
McShane's extension theorem \cite[Theorem~1.33]{Weaver} gives the
extensions
\[
 \widetilde F_m(p)=\inf_{q\in Q}\bigl(g_m(q)+L\,d_M(p,q)\bigr),
 \qquad p\in M.
\]
Here $L\,d_M(p,q)$ is the distance between $p$ and $q$ multiplied by
the constant $L$. The triangle inequality gives
$\widetilde F_m(p)\leqslant\widetilde F_m(p')+L\,d_M(p,p')$;
interchanging $p$ and $p'$ proves the Lipschitz bound $L$.
For $p\in Q$, the Lipschitz bound on $g_m$ gives
$g_m(p)\leqslant g_m(q)+L\,d_M(p,q)$ for every $q\in Q$, and taking
$q=p$ shows that $\widetilde F_m(p)=g_m(p)$.

To retain the uniform bound $\varepsilon_m$, define
\[
 F_m(p)=\max\{-\varepsilon_m,\min\{\varepsilon_m,\widetilde F_m(p)\}\}.
\]
This replaces values above $\varepsilon_m$ by $\varepsilon_m$ and
values below $-\varepsilon_m$ by $-\varepsilon_m$. The scalar map
$s\mapsto\max\{-\varepsilon_m,\min\{\varepsilon_m,s\}\}$ is
$1$-Lipschitz and fixes every point of
$[-\varepsilon_m,\varepsilon_m]$. Since $|g_m|\leqslant\varepsilon_m$
on $Q$, it follows that $F_m$ still extends $g_m$, vanishes at $o$,
and satisfies
$\norm{F_m}_{\mathrm{Lip}}\leqslant L$ and
$\norm{F_m}_\infty\leqslant\varepsilon_m$. In particular, $(F_m)$ is
weak*-null in $\Lipzero(M)$.

For each fixed $t\in K$, the bounded linear map
\[
 T_t:\F(C)\longrightarrow\F(M),\qquad
 T_t\delta_x=\delta_{\Phi(x,t)}-\delta_{\Phi(0,t)}
\]
has $\norm{T_t}\leqslant\norm{\Phi}_{\mathrm{Lip}}$ by the universal
linearisation property of Lipschitz-free spaces
\cite[Theorem~3.6]{Weaver}, and
sends $(\nu_j)$ to a weakly null sequence. Let $(G_k)$ be any convex
block subsequence of $(F_m)$, with notation as in \eqref{eq:blocks}.
Put $b_k=\max E_k$ and $t=t_{\{b_k:k\geqslant1\}}$. For every
$m\in E_k$, the least selected index at least $m$ is $b_k$.
Consequently
\[
 G_k(\Phi(x,t))=H_{j_{b_k}}(x),\qquad
 \langle G_k,T_t\nu_{j_{b_k}}\rangle=1.
\]
Lemma~\ref{lem:diagonal} proves that $\F(M)$ fails (K). Notice that
the same fibre is used for the whole block sequence.

The subspace theorem from \cite{AR} excludes an embedding into an
SWCG space. To exclude an embedding into an arbitrary $L_1$-space,
restrict it to the separable subspace $\F(Q)$. Applying the construction
with ambient space $Q$ shows that $\F(Q)$ fails (K), contradicting
Lemma~\ref{lem:Lone}. Finally, if $M$ is separable then so is $\F(M)$,
and the same lemma gives $\Lipzero(M)\not\cong\ell_\infty$.
\end{proof}

The construction also gives a quantitative estimate. Following
\cite[Section~3]{CKR}, for a bounded sequence $(g_k)\subset X^*$ put
\[
 \beta((g_k))=
 \sup_{\substack{(x_k)\subset B_X\\x_k\to0\ \mathrm{weakly}}}
 \limsup_k|\langle g_k,x_k\rangle|.
\]
The quantity $K_2(X)$ is the supremum, over weak*-null sequences in
$B_{X^*}$, of the infimum of $\beta$ over their convex block
subsequences. If $\Phi$ is as in the proof above and
$D=\norm{\Phi}_{\mathrm{Lip}}\norm{\Phi^{-1}}_{\mathrm{Lip}}$, then
\begin{equation}\label{eq:quantitative}
 K_2(\F(M))\geqslant D^{-1}.
\end{equation}
Indeed, for any $\eta>0$, choose the indices $j_b$ so that
$\norm{H_{j_b}}_\infty\leqslant\eta\varepsilon_b$.
The functions in \eqref{eq:construction} are then
$\eta$-Lipschitz in the second coordinate and have Lipschitz norm
at most $\sqrt{1+\eta^2}$. Run the extension construction with
$L_\eta=\sqrt{1+\eta^2}\norm{\Phi^{-1}}_{\mathrm{Lip}}$
and truncate the extensions to $[-\eta\varepsilon_m,\eta\varepsilon_m]$.
For every convex block sequence $(G_k)$ of these extensions, the
weakly null sequence
\[
 x_k=\frac{T_t\nu_{j_{b_k}}}{\norm{\Phi}_{\mathrm{Lip}}}
 \in B_{\F(M)}
 \quad\text{satisfies}\quad
 \left\langle\frac{G_k}{L_\eta},x_k\right\rangle
 =\frac1{\sqrt{1+\eta^2}D}.
\]
Letting $\eta\downarrow0$ gives \eqref{eq:quantitative}.
Since $K_2(X)\leqslant1$, we obtain $K_2(\F(M))=1$ whenever
the product embeds isometrically.

This proof also explains the significance of positive horizontal
measure. If a non-empty compact set $C\subset I$ has Lebesgue measure
zero, then $\mathds{1}_Cr_j=0$ in $L_1(I)$ and $H_j$ is identically zero.
Moreover, $\F(C)$ has the Schur property: if $C$ is infinite, it is
linearly isometric to $\ell_1$ by \cite[Corollary~3.4]{Godard},
whereas for finite $C$ it is finite-dimensional.
Thus $\F(C)$ contains no normalised weakly null sequence.
This explains why the fibre argument
requires $C$ to have positive measure; it does not show that this condition
is necessary for $\F(C\times K)$ to fail (K).

\begin{proof}[Proof of Theorem~\ref{thm:bundles}]
Each $\Z^k$ is countable and uniformly discrete, so $\F(\Z^k)$ is a
separable Schur space by \cite[Proposition~4.4]{Kalton}; it is therefore
SWCG by Lemma~\ref{lem:SWCG}.
This settles the case $n=0$, including the zero space if $m=0$.
For $n=1$ and $m\geqslant1$, \cite[Proposition~3.14]{AM} gives
\[
 \F(\R\times\Z^m)
 \cong\F(\R)\oplus\F(\Z^{m+1})
 \cong L_1(\R)\oplus\F(\Z^{m+1}),
\]
where the second isomorphism uses \cite[Corollary~3.3]{Godard}.
The first summand is isomorphic to $L_1[0,1]$ and hence is SWCG;
Lemma~\ref{lem:SWCG} applies again. The case $n=1,m=0$ is the same
$L_1$ observation. Every SWCG space has (K) by \cite{AR}.
If $n\geqslant2$, the underlying space contains an isometric copy of
$I^2$, so Theorem~\ref{thm:product} gives failure of (K).
Finally, $\F(\R^d)$ embeds into no SWCG space when $d\geqslant2$,
which proves the final part of the theorem.
\end{proof}

\section{Positive-measure Euclidean sets}

We include the following Eggleston-type argument for completeness;
see \cite{MRZ} for related product theorems and their history.

\begin{lemma}\label{lem:eggleston}
If $E\subset\R\times\R^d$, $d\geqslant1$, is Lebesgue measurable and
has positive measure, then $E$ contains $C\times K$, where $C\subset\R$
is compact with positive measure and $K\subset\R^d$ is a Cantor set.
\end{lemma}

\begin{proof}
By inner regularity, replace $E$ by a compact subset of positive measure.
For $y\in\R^d$, set $E_y=\{x:(x,y)\in E\}$ and
$\Psi(y)=\mathds{1}_{E_y}\in L_1(\R)$. The map $\Psi$ is measurable:
its distances to a countable dense family in $L_1(\R)$ are measurable
by Fubini. Fubini and Lusin give a compact parameter set $P_0$ of
positive measure on which $\Psi$ is continuous and every $E_y$ has
positive measure. Replace $P_0$ by the support of Lebesgue measure
restricted to $P_0$. The resulting compact set $P$ has no isolated
points and retains both properties.

Choose $y_\varnothing\in P$ and let $a=|E_{y_\varnothing}|>0$.
By continuity and the absence of isolated points, choose inductively
$y_s\in P$ and relatively open neighbourhoods $U_s$ for finite binary
words $s$, with $y_s\in U_s$, nested closures, disjoint sibling
closures, diameters tending uniformly to zero, and
\[
 \norm{\Psi(y_{si})-\Psi(y_s)}_1
 <a\,2^{-2|s|-4}\qquad(i=0,1).
\]
Here and below closures are taken in $P$. Then
\[
 C=\bigcap_s E_{y_s},\qquad
 K=\bigcap_{j\geqslant0}\ \bigcup_{|s|=j}\overline{U_s}
\]
are compact, and the nested construction makes $K$ a Cantor set.
Every point lost from the root fibre is lost at some edge of a finite
branch, whence
\[
 |C|\geqslant a-\sum_{s,i}|E_{y_s}\setminus E_{y_{si}}|
 \geqslant a-\sum_{j\geqslant0}2^{j+1}a2^{-2j-4}
 =3a/4>0.
\]
For $y\in K$, the points $y_s$ along its branch converge to $y$.
Closedness of $E$ then gives $C\times\{y\}\subset E$.
\end{proof}

\begin{proof}[Proof of Theorem~\ref{thm:measure}]
Apply Lemma~\ref{lem:eggleston} to $\overline E$, with $d=n-1$,
and then Theorem~\ref{thm:product}. Since
$\F(E)=\F(\overline E)$ canonically, all the conclusions follow.
\end{proof}

Aliaga \cite[Theorem~3.7 and Corollary~3.8]{Aliaga} proves
$\Lipzero(E)\cong\Lipzero(\R^n)$ whenever $E\subset\R^n$ is non-porous,
in particular whenever it has positive measure. That dual isomorphism
alone does not identify the weak* topologies determined by the
canonical preduals. Theorem~\ref{thm:measure} establishes failure of
(K) directly for those preduals.
Aliaga's subsequent theorem \cite[Theorem~1]{AliagaFree} gives a
complemented copy of $\F(\Z^n)$ in $\F(E)$ whenever $E$ is non-porous.
It does not yield failure of (K), since $\F(\Z^n)$ is Schur.
Indeed, the positive-measure hypothesis in Theorem~\ref{thm:measure}
cannot be replaced by non-porosity: $\Z^n$ is non-porous in the
all-scales sense of \cite[Definition~2.1]{Aliaga}, but its Lipschitz-free
space is SWCG. That definition requires a fixed $\lambda>0$ such
that every ambient ball $B(z,R)$, at every scale $R>0$, contains
a ball of radius $\lambda R$ disjoint from the set.
Every point of $\R^n$ is within $\sqrt n/2$ of $\Z^n$, so every ball
of radius greater than $\sqrt n/2$ meets $\Z^n$. Taking
$\lambda R>\sqrt n/2$ rules out the required hole. This is compatible
with the porosity of $\Z^n$ at small scales.

Candido, C\'uth and Doucha \cite[Theorem~1.16]{CCD} proved
$\Lipzero(\Z^d)\cong\Lipzero(\R^d)$ for every $d\geqslant1$.
The non-isomorphism of the canonical preduals was already known:
$\F(\Z^d)$ is Schur, whereas $\F(\R^d)$ contains a copy of $L_1$
coming from a line. For $d\geqslant2$, the present argument also
distinguishes these preduals by property (K).

\section{An example of Assouad dimension one}

For a metric space $M$, its Assouad dimension is the infimum
of the exponents $s$ for which there is $A_s<\infty$ such that
\[
 N_r(B(p,R)\cap M)\leqslant A_s(R/r)^s
 \qquad(p\in M,\ 0<r<R\leqslant\diam M),
\]
where $\diam M$ may be infinite and $N_r$ is the least number of
balls of radius $r$ needed to cover
the indicated set.

\begin{proof}[Proof of Theorem~\ref{thm:assouad}]
Starting with $[0,1]$, at stage $j\geqslant1$ retain the two endpoint
subintervals of relative length $a_j=2^{-j-2}$ in each remaining
interval. Let $K_*$ be the resulting Cantor set. The intervals at level
$j$ have length
$L_j=\prod_{i=1}^j a_i=2^{-j(j+5)/2}$, with $L_0=1$.

We verify that $\dimA K_*=0$. Fix $\varepsilon>0$ and choose $J$ so that
$a_j^\varepsilon\leqslant1/2$ for $j\geqslant J$. For $0<r<R<1$, choose
$k,n$ with $L_k\leqslant R<L_{k-1}$ and
$L_n\leqslant r<L_{n-1}$. Then $n\geqslant k$. Distinct level-$(k-1)$
intervals are separated by gaps of length at least $6L_{k-1}>2R$;
for $k=1$ there is just one such interval. Thus a ball of radius $R$
meets at most one, and
$N_r(B(p,R)\cap K_*)\leqslant2^{n-k+1}$. If $n\leqslant k+1$, this
is at most $4$. Otherwise,
\[
 2^{n-k-1}\leqslant 2^J
 \left(\prod_{j=k+1}^{n-1}a_j^{-1}\right)^\varepsilon
 =2^J\left(\frac{L_k}{L_{n-1}}\right)^\varepsilon
 \leqslant2^J(R/r)^\varepsilon.
\]
The same global estimate treats $R=1$, so $\dimA K_*=0$. Product covers
give $\dimA(I\times K_*)\leqslant1$, and the contained interval gives
the reverse inequality. Set $M=I\times K_*$ and apply
Theorem~\ref{thm:product}.

Since $M$ is a compact subset of the plane, \cite{KoP} gives $(V^*)$
and weak sequential completeness, while \cite[Corollary~2]{Mason}
gives the Dunford--Pettis property.

Finally, suppose that $f:\{0,\ldots,N\}^2\to M$ satisfies
\[
 \lambda|p-q|\leqslant d(f(p),f(q))\leqslant D\lambda|p-q|
 \qquad(p,q\in\{0,\ldots,N\}^2)
\]
for some $\lambda>0$ and fixed $D$. Its $(N+1)^2$ image points are
$\lambda$-separated and lie in a ball of radius $\sqrt2D\lambda N$.
The Assouad estimate with exponent $3/2$ and $r=\lambda/3$ gives
\[
 (N+1)^2\leqslant A(3\sqrt2DN)^{3/2}.
\]
If the displayed outer radius exceeds $\diam M$, use $\diam M$
instead. The constant $A$ depends only on $M$, so $N$ is bounded for
each fixed $D$.
\end{proof}

The dimension in Theorem~\ref{thm:assouad} is minimal even among
complete metric spaces. If $\dimA M<1$, every bounded ball in $M$
has zero $\mathcal H^1$-measure, and a countable covering by such
balls gives $\mathcal H^1(M)=0$. Thus $M$ is purely
$1$-unrectifiable, and \cite[Theorem~C]{AGPP} gives the Schur
property and (K).

For context, Kochanek and Perneck\'a \cite{KoP} establish $(V^*)$
and weak sequential completeness for Lipschitz-free spaces over
compact subsets of superreflexive spaces. Aliaga, Perneck\'a and
Quero \cite[Theorem~C]{APQ} prove $(V^*)$ over Hilbert spaces;
as this property passes to closed subspaces, it holds for all our
Euclidean examples. More generally, Jung, Petitjean, Proch\'azka
and Quilis \cite[Proposition~6.1(ii)]{JPPQ} obtain $(V^*)$ for
metric spaces of finite Nagata dimension, including doubling spaces.
Mason's theorem \cite[Corollary~2]{Mason} gives the Dunford--Pettis
property for every subset of a finite-dimensional Euclidean space.
The example above therefore has all three properties despite failing (K).

\section{Further questions}

Let $\Delta$ be the middle-third Cantor set. Both $I\times\Delta$ and
the product in Theorem~\ref{thm:assouad} are porous in the plane.
Indeed, all retention ratios are at most $1/3$. For a point of either
Cantor set and $0<r<4$, the first level interval of length at most $r/4$
has a parent with a middle gap longer than $r/12$, at distance at most
$r/4$ from the point. This gives a complementary planar ball of radius
$r/48$ inside the ball of radius $r$ centred at the corresponding
product point. Larger scales follow from boundedness. For an arbitrary
ambient ball $B(z,R)$, either $B(z,R/2)$ misses the product or it contains
a product point to which this construction applies with radius $R/2$.
Thus one may take porosity constant $1/96$.

Porosity, rather than measure zero alone, places these products outside
Aliaga's non-porosity theorem. We know that
$\Lipzero(I\times\Delta)\not\cong\ell_\infty$,
but the argument does not decide whether
$\Lipzero(I\times\Delta)\cong\Lipzero(I^2)$. A negative answer would
disprove the line-or-plane dichotomy in \cite[Question~3]{Aliaga};
a positive answer would supply a porous example for
\cite[Question~4]{Aliaga}.
The recent theorem of Aliaga, C\'uth and Vico
\cite[Theorem~1.1]{ACV} makes every Lipschitz-free space the strongly
unique \emph{isometric} predual of its Lipschitz space. This does not
require arbitrary Banach space isomorphisms of Lipschitz spaces to
preserve their canonical weak* topologies, and therefore does not
settle the question above. The example $\Z^2$ versus $\R^2$ already
shows that such an inference would fail.

\begin{question}
If $M\subset\R^2$ is compact and $\F(M)$ fails (K), must $M$ contain
a bi-Lipschitz copy of $C\times K$ as in Theorem~\ref{thm:product}?
\end{question}

Compactness is a substantive restriction here: a countable dense
subset of the square has the same Lipschitz-free space as the square,
but cannot contain such an uncountable product.

\begin{question}
Can the construction give explicit lower bounds for the $L_1$-distortion
of finite-dimensional subspaces of $\F(I\times K)$ in terms of the
separation scales of $K$?
\end{question}

The estimate \eqref{eq:quantitative} quantifies failure of (K), but
does not by itself give finite-dimensional $L_1$-distortion bounds.
In particular, it does not recover the grid estimates of \cite{NS}.

\section*{Use of artificial intelligence}

This project originated in 2019. ChatGPT (OpenAI) was subsequently
used to assess whether our preliminary results could address more
recent problems, notably Aliaga's question concerning the isomorphism
of $\F(\R\times\Z)$ and $\F(\R^2)$
\cite[Question~5]{Aliaga}, also recorded in
\cite[Problem~2]{Mason}, and his questions about the isomorphism
types of Lipschitz spaces over planar sets and the r\^ole of porosity
\cite[Questions~3 and~4]{Aliaga}.
It also assisted with literature searches and the revision of the
exposition. We retain
responsibility for the mathematical content and references.

\end{document}